\documentclass[reqno,12pt]{amsart}

\usepackage{amsmath, epsfig, cite}
\usepackage{amssymb}
\usepackage{amsfonts}
\usepackage{latexsym}
\usepackage{amsthm}

\newtheorem{theorem}{Theorem}[section]

\newtheorem{corollary}[theorem]{Corollary}

\theoremstyle{remark}

\newcommand{\N}{\mathbb N}

\author{Victor J.\ W.\ Guo}
\address{School of Mathematical Sciences, Huaiyin Normal University,
Huai'an 223300, Jiangsu, People's Republic of China}
\email{jwguo@hytc.edu.cn}
\thanks{The first author was partially supported by the National Natural
Science Foundation of China (grant 11771175).}

\author{Michael J.\ Schlosser}
\address{Fakult\"at f\"ur Mathematik, Universit\"at Wien,
Oskar-Morgenstern-Platz~1, A-1090 Vienna, Austria}
\email{michael.schlosser@univie.ac.at}

\title[Proof of a $q$-supercongruence modulo $\Phi_n(q)^5$]{Proof of a basic hypergeometric supercongruence modulo the fifth power of a cyclotomic polynomial}



\begin{document}

\begin{abstract}
By means of the $q$-Zeilberger algorithm,
we prove a basic hypergeometric supercongruence modulo the fifth power
of the cyclotomic polynomial $\Phi_n(q)$. This result appears to be quite unique,
as in the existing literature so far no basic hypergeometric supercongruences
modulo a power greater than the fourth of a cyclotomic polynomial have been proved.
We also establish a couple of related results, including a parametric supercongruence.
\end{abstract}

\maketitle

\section{Introduction}\label{secintro}
In 1997, Van Hamme \cite{Hamme} conjectured that 13 Ramanujan-type series including
\begin{align*}
\sum_{k=0}^\infty (-1)^k(4k+1)\frac{(\frac{1}{2})_k^3}{k!^3 }=\frac{2}{\pi}, 
\end{align*}
admit nice $p$-adic analogues, such as
\begin{align*}
\sum_{k=0}^{\frac{p-1}2} (-1)^k(4k+1)\frac{(\frac{1}{2})_k^3}{k!^3 }\equiv
p\,(-1)^{\frac{p-1}2}\pmod{p^3},
\end{align*}
where $(a)_n=a(a+1)\cdots(a+n-1)$ denotes the Pochhammer symbol and $p$ is an odd prime. Up to present,
all of the 13 supercongruences have been confirmed. See \cite{OZ,Swisher} for historic remarks on these supercongruences.
Recently, $q$-analogues of congruences and supercongruences have caught the interests of many authors (see, for example,
\cite{Gorodetsky,Guillera,Guo1,Guo2,Guo3,Guo4,Guo4.5,Guo5,Guo-jmaa,Guo6,GN,GS0,GS,GW0,GW,GuoZu,GuoZu2,LPZ,LP,NP,Tauraso1,Tauraso2,Zudilin}).
In particular, the first author and Zudilin \cite{GuoZu} devised a method, called
`creative microscoping', to prove quite a few $q$-supercongruences by introducing an additional parameter $a$.
In \cite{GS}, the authors of the present paper proved many additional $q$-supercongruences by the creative microscoping method.
Supercongruences modulo a higher integer power of a prime, or, in the $q$-case,
of a cyclotomic polynomial, are very special and usually difficult to prove.
As far as we know, until now the result
\begin{equation}\label{eq:GW}
\sum_{k=0}^{\frac{n-1}2}[4k+1]\frac{(q;q^2)_k^4}{(q^2;q^2)_k^4}
\equiv q^{\frac{1-n}2}[n]+\frac{(n^2-1)(1-q)^2}{24}q^{\frac{1-n}2}[n]^3
\pmod{[n]\Phi_n(q)^3},
\end{equation}
for an odd positive integer $n$, due to the first author and Wang~\cite{GW},
is the unique $q$-supercongruence modulo $[n]\Phi_n(q)^3$ in the literature
that was completely proved. (Several similar conjectural $q$-supercongruences
are stated in \cite{GS} and in \cite{GuoZu}.)
The purpose of this paper is to establish an even higher $q$-congruence,
namely modulo a fifth power of a cyclotomic polynomial.
Specifically, we prove the following three theorems.
(The first two together confirm a conjecture by the authors \cite[Conjecture 5.4]{GS}).

\begin{theorem}\label{thm:main-1}
Let $n>1$ be a positive odd integer. Then
\begin{subequations}
\begin{align}
\sum_{k=0}^{\frac{n+1}2}[4k-1]\frac{(q^{-1};q^2)_k^4}{(q^2;q^2)_k^4} q^{4k}
&\equiv -(1+3q+q^2)[n]^4 \pmod{[n]^4\Phi_n(q)},  \label{eq:main-1}
\intertext{and}
\sum_{k=0}^{n-1}[4k-1]\frac{(q^{-1};q^2)_k^4}{(q^2;q^2)_k^4} q^{4k}
&\equiv -(1+3q+q^2)[n]^4 \pmod{[n]^4\Phi_n(q)}.  \label{eq:main-2}
\end{align}
\end{subequations}
\end{theorem}

\begin{theorem}\label{thm:main-2}
Let $n>1$ be a positive odd integer. Then
\begin{align*}
\sum_{k=0}^{\frac{n+1}2}[4k-1]\frac{(aq^{-1};q^2)_k(q^{-1}/a;q^2)_k(q^{-1};q^2)_k^2}
{(aq^2;q^2)_k(q^2/a;q^2)_k(q^2;q^2)_k^2} q^{4k}
&\equiv 0 \pmod{[n]^2(1-aq^n)(a-q^n)},  \\
\intertext{and}
\sum_{k=0}^{n-1}[4k-1]\frac{(aq^{-1};q^2)_k(q^{-1}/a;q^2)_k(q^{-1};q^2)_k^2}
{(aq^2;q^2)_k(q^2/a;q^2)_k(q^2;q^2)_k^2} q^{4k}
&\equiv 0 \pmod{[n]^2(1-aq^n)(a-q^n)}.
\end{align*}
\end{theorem}

The $a=-1$ case of Theorem \ref{thm:main-2} admits an even stronger $q$-congruence.

\begin{theorem}\label{thm:main-3}
Let $n>1$ be a positive odd integer. Then
\begin{subequations}
\begin{align}
\sum_{k=0}^{\frac{n+1}2}[4k-1]\frac{(q^{-2};q^4)_k^2}{(q^4;q^4)_k^2}q^{4k}
&\equiv -q^n(1-q+q^2) [n]_{q^2}^2\pmod{[n]_{q^2}^2\Phi_n(q^2)},  \label{eq:3-1} \\
\intertext{and}
\sum_{k=0}^{n-1}[4k-1]\frac{(q^{-2};q^4)_k^2}{(q^4;q^4)_k^2}q^{4k}
&\equiv -(1-q+q^2) [n]_{q^2}^2\pmod{[n]_{q^2}^2\Phi_n(q^2)}.  \label{eq:3-1b}
\end{align}
\end{subequations}
\end{theorem}

In the above $q$-supercongruences and in what follows,
\begin{equation*}
(a;q)_n=(1-a)(1-aq)\cdots (1-aq^{n-1})
\end{equation*}
is the $q$-shifted factorial,
\begin{equation*}
[n]=[n]_q=1+q+\cdots+q^{n-1}
\end{equation*}
is the $q$-number,
\begin{equation*}
\begin{bmatrix}n\\k\end{bmatrix}=\begin{bmatrix}n\\k\end{bmatrix}_q:=
\frac{(q;q)_n}{(q;q)_k(q;q)_{n-k}}
\end{equation*}
is the $q$-binomial coefficient,
and $\Phi_n(q)$ is the $n$-th cyclotomic polynomial of $q$.
Note that the congruences in Theorem \ref{thm:main-1} modulo
$[n]\Phi_n(q)^2$ and the congruences in Theorem \ref{thm:main-2}
modulo $[n](1-aq^n)(a-q^n)$ have already been proved by the authors
in \cite[eqs.\ (5.5) and (5.10)]{GS}.

\section{Proof of Theorem \ref{thm:main-1} by the Zeilberger algorithm}
The Zeilberger algorithm (cf.\ \cite{PWZ}) can be used to find that the functions
\begin{align*}
f(n,k)&=(-1)^k \frac{(4n-1)(-\frac{1}{2})_n^3 (-\frac{1}{2})_{n+k}}
{(1)_n^3 (1)_{n-k}(-\frac{1}{2})_k^2},\\[5pt]
g(n,k)&=(-1)^{k-1}\frac{4(-\frac{1}{2})_n^3 (-\frac{1}{2})_{n+k-1}}
{(1)_{n-1}^3 (1)_{n-k}(-\frac{1}{2})_k^2}
\end{align*}
satisfy the relation
\begin{align*}
(2k-3)f(n,k-1)-(2k-4)f(n,k)=g(n+1,k)-g(n,k).
\end{align*}
Of course, given this relation, it is not difficult to verify by hand that it is
satisfied by the above
pair of doubly-indexed sequences $f(n,k)$ and $g(n,k)$.

Here we use the convention $1/(1)_m=0$ for all negative integers $m$.
We now define the $q$-analogues of $f(n,k)$ and $g(n,k)$ as follows:
\begin{align*}
F(n,k) &=(-1)^{k}q^{(k-2)(k-2n+1)}\frac{[4n-1](q^{-1};q^2)_{n}^3(q^{-1};q^2)_{n+k}}{(q^2;q^2)_{n}^3(q^2;q^2)_{n-k}(q^{-1};q^2)_{k}^2}, \\[5pt]
G(n,k) &=\frac{(-1)^{k-1}q^{(k-2)(k-2n+3)}(q^{-1};q^2)_{n}^3(q^{-1};q^2)_{n+k-1}}{(1-q)^2(q^2;q^2)_{n-1}^3(q^2;q^2)_{n-k}(q^{-1};q^2)_{k}^2},
\end{align*}
where we have used the convention that $1/(q^2;q^2)_{m}=0$ for $m=-1,-2,\ldots.$
Then the functions $F(n,k)$ and $G(n,k)$ satisfy the relation
\begin{equation}
[2k-3]F(n,k-1)-[2k-4]F(n,k)=G(n+1,k)-G(n,k).  \label{eq:fnk-gnk}
\end{equation}
Indeed, it is straightforward to obtain the following expressions:
\begin{align*}
\frac{F(n,k-1)}{G(n,k)} &=\frac{q^{2n-4k+6}(1-q)(1-q^{4n-1})(1-q^{2k-3})^2}{(1-q^{2n-2k+2})(1-q^{2n})^3},\\[5pt]
\frac{F(n,k)}{G(n,k)}   &=-\frac{q^{4-2k}(1-q)(1-q^{4n-1})(1-q^{2n+2k-3})}{(1-q^{2n})^3},\\[5pt]
\frac{G(n+1,k)}{G(n,k)} &=\frac{q^{4-2k}(1-q^{2n-1})^3(1-q^{2n+2k-3})}{(1-q^{2n})^3(1-q^{2n-2k+2})}.
\end{align*}
It is easy to verify the identity
\begin{align*}
&\frac{q^{2n-4k+6}(1-q^{4n-1})(1-q^{2k-3})^3}{(1-q^{2n-2k+2})(1-q^{2n})^3}
+\frac{q^{4-2k}(1-q^{2k-4})(1-q^{4n-1})(1-q^{2n+2k-3})}{(1-q^{2n})^3} \\[5pt]
&\quad =\frac{q^{4-2k}(1-q^{2n-1})^3(1-q^{2n+2k-3})}{(1-q^{2n})^3(1-q^{2n-2k+2})}-1,
\end{align*}
which is equivalent to \eqref{eq:fnk-gnk}.
(Alternatively, we could have established \eqref{eq:fnk-gnk} by only guessing
$F(n,k)$ and invoking the $q$-Zeilberger algorithm \cite{WZ}.)

Let $m>1$ be an odd integer. Summing \eqref{eq:fnk-gnk} over $n$ from $0$ to $(m+1)/2$, we get
\begin{align}
[2k-3]\sum_{n=0}^{\frac{m+1}2}F(n,k-1)-[2k-4]\sum_{n=0}^{\frac{m+1}2}F(n,k)
&=G\left(\frac{m+3}{2},k\right)-G(0,k) \notag\\[5pt]
&=G\left(\frac{m+3}{2},k\right).  \label{eq:fnk-gn0-00}
\end{align}
We readily compute
\begin{subequations}\label{eq:gppks}
\begin{align}
G\left(\frac{m+3}{2},1\right)
&=\frac{q^{m-1} (q^{-1};q^2)_{(m+3)/2}^4}{(1-q)^2(q^2;q^2)_{(m+1)/2}^4(1-q^{-1})^2} \notag \\[5pt]
&=\frac{q^{m-3}[m]^4}{[m+1]^4 (-q;q)_{(m-1)/2}^8}
\begin{bmatrix}m-1\\(m-1)/2\end{bmatrix}^4, \label{eq:gppk}  \\
\intertext{and}
G\left(\frac{m+3}{2},2\right)
&=-\frac{(q^{-1};q^2)_{(m+3)/2}^3(q^{-1};q^2)_{(m+5)/2}}{(1-q)^2(q^2;q^2)_{(m+1)/2}^3(q^2;q^2)_{(m-1)/2}(q^{-1};q^2)_2^2}  \notag\\[5pt]
&=-\frac{q^{-2} [m]^4 [m+2]}{[m+1]^3 (-q;q)_{(m-1)/2}^8}
\begin{bmatrix}m-1\\(m-1)/2\end{bmatrix}^4.  \label{eq:gppk-2}
\end{align}
\end{subequations}
Combining \eqref{eq:fnk-gn0-00} and \eqref{eq:gppks}, we have
\begin{align*}
\sum_{n=0}^{\frac{m+1}2}F(n,0)&
=\frac{[-2]}{[-1]}\sum_{n=0}^{\frac{m+1}2}F(n,1)+
\frac{1}{[-1]}G\left(\frac{m+3}{2},1\right) \\[5pt]
&=\frac{1+q}{q}G\left(\frac{m+3}{2},2\right)-
q G\left(\frac{m+3}{2},1\right)   \\[5pt]
&=-\frac{(1+q) [m]^4 [m+1][m+2]+q^{m+1}[m]^4}
{q^3 [m+1]^4 (-q;q)_{(m-1)/2}^8}\begin{bmatrix}m-1\\(m-1)/2\end{bmatrix}^4,
\end{align*}
i.e.,
\begin{align}
\sum_{n=0}^{\frac{m+1}2}[4n-1]\frac{(q^{-1};q^2)_n^4}{(q^2;q^2)_n^4} q^{4n}
=-\frac{(1+q) [m]^4 [m+1][m+2]+q^{m+1}[m]^4}{q [m+1]^4 (-q;q)_{(m-1)/2}^8}\begin{bmatrix}m-1\\(m-1)/2\end{bmatrix}^4.  \label{eq:sum-4k-1}
\end{align}
By \cite[Lemma 2.1]{Guo2} (or \cite[Lemma 2.1]{Guo1}), we have $(-q;q)_{(m-1)/2}^2\equiv q^{(m^2-1)/8}$ $\pmod{\Phi_m(q)}$. Moreover, it is easy to see that
$$
\begin{bmatrix}m-1\\(m-1)/2\end{bmatrix}
=\prod_{k=1}^{(m-1)/2}\frac{1-q^{m-k}}{1-q^k}\equiv \prod_{k=1}^{(m-1)/2}\frac{1-q^{-k}}{1-q^k}=(-1)^{(m-1)/2}q^{(1-m^2)/8}\pmod{\Phi_m(q)},
$$
and $[m]$ is relatively prime to $(-q;q)_{(m-1)/2}$. It follows from \eqref{eq:sum-4k-1} that
\begin{align*}
\sum_{n=0}^{\frac{m+1}2}[4n-1]\frac{(q^{-1};q^2)_n^4}{(q^2;q^2)_n^4} q^{4n}
\equiv -((1+q)^2+q)[m]^4 \pmod{[m]^4\Phi_m(q)}.
\end{align*}
Concluding, the congruence \eqref{eq:main-1} holds.

Similarly, summing \eqref{eq:fnk-gnk} over $n$ from $0$ to $m-1$, we get
\begin{align*}
[2k-3]\sum_{n=0}^{m-1}F(n,k-1)-[2k-4]\sum_{n=0}^{m-1}F(n,k)
=G(m,k),
\end{align*}
and so
\begin{align}
\sum_{n=0}^{m-1}[4n-1]\frac{(q^{-1};q^2)_n^4}{(q^2;q^2)_n^4} q^{4n}
&=\frac{1+q}{q}G(m,2)-q G(m,1) \notag \\[5pt]
&=-\frac{(1+q) [2m-2][2m-1]+q^{2m-2}}{q (-q;q)_{m-1}^8}
\begin{bmatrix}2m-2\\m-1\end{bmatrix}^4. \label{eq:mmm}
\end{align}
It is easy to see that
$$
\frac{1}{[m]}\begin{bmatrix}2m-2\\m-1\end{bmatrix}
=\frac{1}{[m-1]}\begin{bmatrix}2m-2\\m-2\end{bmatrix}
\equiv (-1)^{m-2}q^{2-\binom{m-1}2}\pmod{\Phi_m(q)},
$$
and $(-q;q)_{m-1}\equiv 1\pmod{\Phi_m(q)}$ (see, for example, \cite{Guo2}).
The proof of \eqref{eq:main-2} then follows easily from \eqref{eq:mmm}.

\section{Proof of Theorems \ref{thm:main-2} and \ref{thm:main-3}}
\begin{proof}[Proof of Theorem \ref{thm:main-2}]
It is easy to see by induction on $N$ that
\begin{align}
&\sum_{k=0}^{N}[4k-1]\frac{(aq^{-1};q^2)_k(q^{-1}/a;q^2)_k(q^{-1};q^2)_k^2}
{(aq^2;q^2)_k(q^2/a;q^2)_k(q^2;q^2)_k^2} q^{4k} \notag\\[5pt]
&\quad=\frac{(aq;q^2)_N (q/a;q^2)_N ((a+1)^2q^{2N+1}-a(1+q)(1+q^{4N+1}))}{q(a-q)(1-aq)(aq^2;q^2)_N (q^2/a;q^2)_N (-q;q)_N^4}\begin{bmatrix}2N\\N\end{bmatrix}^2. \label{eq:sum-an}
\end{align}
For $N=(n+1)/2$ or $N=n-1$, we see that $(aq;q^2)_N (q/a;q^2)_N$ contains the factor $(1-aq^n)(1-q^n/a)$. Moreover,
\begin{align*}
\frac{[(n+1)/2]}{[n]}\begin{bmatrix}n\\(n-1)/2\end{bmatrix}=
\begin{bmatrix}n-1\\(n-1)/2\end{bmatrix}
\end{align*}
is a polynomial in $q$. Since $[(n+1)/2]$ and $[n]$ are relatively prime,
we conclude that $\begin{bmatrix}\begin{smallmatrix}n\\(n-1)/2\end{smallmatrix}\end{bmatrix}$ is divisible by $[n]$.
Therefore, $\begin{bmatrix}\begin{smallmatrix}n+1\\(n+1)/2\end{smallmatrix}\end{bmatrix}
=(1+q^{(n+1)/2})\begin{bmatrix}\begin{smallmatrix}n\\(n-1)/2\end{smallmatrix}\end{bmatrix}$
is also divisible by $[n]$. It is also well known that
$\begin{bmatrix}\begin{smallmatrix}2n-2\\n-1\end{smallmatrix}\end{bmatrix}$ is divisible by $[n]$.
Moreover, it is easy to see that $[n]$ is relatively prime to $1+q^m$
for any non-negative integer $m$.
The proof then follows from \eqref{eq:sum-an} by taking $N=(n+1)/2$ and $N=n-1$.
\end{proof}

\begin{proof}[Proof of Theorem \ref{thm:main-3}]
For $a=-1$, the identity \eqref{eq:sum-an} reduces to
\begin{align}
\sum_{k=0}^{N}[4k-1]\frac{(q^{-2};q^4)_k^2}
{(q^4;q^4)_k^2} q^{4k}
&=-\frac{(-q;q^2)_N^2 (1+q^{4N+1})}{q(1+q)(-q^2;q^2)_N^2 (-q;q)_N^4}
\begin{bmatrix}2N\\N\end{bmatrix}^2  \notag\\[5pt]
&=-\frac{(1+q^{4N+1})}{q(1+q)(-q^2;q^2)_N^4 }\begin{bmatrix}2N\\N\end{bmatrix}_{q^2}^2
\label{eq:sum-an-2}
\end{align}
Note that, in the proof of Theorem \ref{thm:main-2}, we have proved that $\begin{bmatrix}\begin{smallmatrix}2N\\N\end{smallmatrix}\end{bmatrix}_{q^2}$ is divisible by $[n]_{q^2}$ for both $N=(n+1)/2$ and $N=n-1$.
Moreover, $[n]_{q^2}$ is relatively prime to $(-q^2;q^2)_m$ for $m\geqslant 0$.
Hence the right-hand side of \eqref{eq:sum-an-2}
is congruent to $0$ modulo $[n]_{q^2}^2$ for $N=(n+1)/2$ or $N=n-1$. To further determine the right-hand side of \eqref{eq:sum-an-2} modulo $[n]_{q^2}^2\Phi_n(q^2)$,
we need only to use the same congruences (with $q\mapsto q^2$) used in the proof of Theorem \ref{thm:main-1}.
\end{proof}

\section{Immediate consequences}

Notice that for $n=p^r$ being an odd prime power,
$\Phi_{p^r}(q)=[p]_{q^{p^{r-1}}}$ holds.
This observation was used in \cite{GW} to extend \eqref{eq:GW}
to a supercongruence modulo $[p^r][p]_{q^{p^{r-1}}}^3$.
In the same vein we immediately deduce from Theorem~\ref{thm:main-1}
the following result:

\begin{corollary}
Let $p$ be an odd prime and $r$ a positive integer. Then
\begin{subequations}
\begin{align}
\sum_{k=0}^{\frac{p^r+1}2}[4k-1]\frac{(q^{-1};q^2)_k^4}{(q^2;q^2)_k^4} q^{4k}
&\equiv -(1+3q+q^2)[p^r]^4 \pmod{[p^r]^4[p]_{q^{p^{r-1}}}},  \label{eq:cor-1a}
\intertext{and}
\sum_{k=0}^{p^r-1}[4k-1]\frac{(q^{-1};q^2)_k^4}{(q^2;q^2)_k^4} q^{4k}
&\equiv -(1+3q+q^2)[p^r]^4 \pmod{[p^r]^4[p]_{q^{p^{r-1}}}}.  \label{eq:cor-1b}
\end{align}
\end{subequations}
\end{corollary}
The $q\to 1$ limiting cases of these two identities
yield the following supercongruences:
\begin{corollary}\label{corr-1}
Let $p$ be an odd prime and $r$ a positive integer. Then
\begin{subequations}
\begin{align}
\sum_{k=0}^{\frac{p^r-1}2}\frac{4k+3}{16(k+1)^4\,256^k}{\binom{2k}k}^4
&\equiv 1-5p^{4r}\pmod{p^{4r+1}},  \label{eq:ccor-1a}
\intertext{and}
\sum_{k=0}^{p^r-2}\frac{4k+3}{16(k+1)^4\,256^k}{\binom{2k}k}^4
&\equiv 1-5p^{4r}\pmod{p^{4r+1}}. \label{eq:ccor-1b}
\end{align}
\end{subequations}
\end{corollary}

Similarly, we deduce from Theorem~\ref{thm:main-3}
the following result:
\begin{corollary}
Let $p$ be an odd prime and $r$ a positive integer. Then
\begin{subequations}
\begin{align}
\sum_{k=0}^{\frac{p^r+1}2}[4k-1]\frac{(q^{-2};q^4)_k^2}{(q^4;q^4)_k^2}q^{4k}
&\equiv -q^{p^r}(1-q+q^2) [p^r]_{q^2}^2\pmod{[p^r]_{q^2}^2[p]_{q^{2p^{r-1}}}},  \label{eq:cor-2a} \\
\intertext{and}
\sum_{k=0}^{p^r-1}[4k-1]\frac{(q^{-2};q^4)_k^2}{(q^4;q^4)_k^2}q^{4k}
&\equiv -(1-q+q^2) [p^r]_{q^2}^2\pmod{[p^r]_{q^2}^2[p]_{q^{2p^{r-1}}}}.  \label{eq:cor-2b}
\end{align}
\end{subequations}
\end{corollary}
The $q\to 1$ limiting cases of these two identities
yield the following supercongruences:
\begin{corollary}\label{corr-2}
Let $p$ be an odd prime and $r$  a positive integer. Then
\begin{subequations}
\begin{align}
\sum_{k=0}^{\frac{p^r-1}2}\frac{4k+3}{4(k+1)^2\,16^k}{\binom{2k}k}^2
&\equiv 1-p^{2r}\pmod{p^{2r+1}},  \label{eq:ccor-2a} \\
\intertext{and}
\sum_{k=0}^{p^r-2}\frac{4k+3}{4(k+1)^2\,16^k}{\binom{2k}k}^2
&\equiv 1-p^{2r}\pmod{p^{2r+1}}. \label{eq:ccor-2b}
\end{align}
\end{subequations}
\end{corollary}
The supercongruences in Corollaries~\ref{corr-1} and \ref{corr-2} are remarkable
since they are valid for arbitrarily high prime powers.
Swisher~\cite{Swisher} had empirically observed several similar but different hypergeometric supercongruences and stated them without proof.



\end{document}